\documentclass[10pt,reqno]{amsart}

\usepackage[latin1]{inputenc}
\usepackage{amsmath}
\usepackage{amssymb}
\usepackage{hyperref}
\usepackage{stackrel}
\usepackage{xcolor}
\usepackage{graphicx}
\usepackage{subfigure}
\usepackage{comment}

\newtheorem{theorem}{Theorem}[section]
\newtheorem{lemma}[theorem]{Lemma}

\newtheorem{definition}[theorem]{Definition}

\newtheorem{proposition}[theorem]{Proposition}
\newtheorem{remark}[theorem]{Remark}

\DeclareMathOperator{\ad}{ad}

\DeclareMathOperator{\inter}{int}
\DeclareMathOperator{\spann}{span}

\DeclareMathOperator{\ri}{ri}
\DeclareMathOperator{\aff}{aff}

\newcommand{\sfe}{S_{F_e}}

\makeatletter
\@namedef{subjclassname@2020}{\textup{2020} Mathematics Subject Classification}
\makeatother

\begin{document}

\title[Extremals on Lie groups]{Extremals on Lie groups with asymmetric polyhedral Finsler structures}

\author[J. B. Prudencio and R. Fukuoka]{J. B. Prudencio and R. Fukuoka \\ \\ \scriptsize{Department of Mathematics, State University of Maring\'a, 87020-900, Maring\'a, PR, Brazil}}

\address{Emails: jbuzattop@gmail.com; rfukuoka@uem.br}

\date{\today}

\begin{abstract}
In this work we study extremals on  Lie groups $G$ endowed with a left invariant polyhedral Finsler structure. We use the Pontryagin's Maximal Principle (PMP) to find curves on the cotangent bundle of the group, such that its projections on $G$ are extremals. 
	Let $\mathfrak g$ and $\mathfrak g^\ast$ be the Lie algebra of $G$ and its dual space respectively.
	We represent this problem as a control system $\mathfrak a^\prime(t)= -\ad^\ast(u(t))(\mathfrak a(t))$ of Euler-Arnold type equation, where $u(t)$ is a measurable control in the unit sphere of $\mathfrak g$ and $\mathfrak a(t)$ is an absolutely continuous curve in $\mathfrak g^\ast$.
	A solution $(u(t), \mathfrak a(t))$ of this control system is a Pontryagin extremal and $\mathfrak a(t)$ is its vertical part.
	In this work we show that for a fixed vertical part of the Pontryagin extremal $\mathfrak a(t)$, the uniqueness of $u(t)$ such that $(u(t),\mathfrak a(t))$ is a Pontryagin extremal can be studied through  an asymptotic curvature of $\mathfrak a(t)$.
\end{abstract}

\keywords{Lie groups, polyhedral Finsler structures, asymptotic curvature, Pontryagin extremals, uniqueness.}

\subjclass[2020]{49N99, 53B20, 53B99, 53D25}

\maketitle

\section{Introduction}

Let $M$ be a differentiable manifold, $T_xM$ the tangent space of $M$ at $x$ and $T^\ast_xM$ the cotangent space of $M$ at $x$.
Let $TM = \{(x,y); \ x\in M, \ y\in T_xM\}$ and $T^\ast M = \{(x,\xi); \ x \in M, \ \xi \in T^\ast_xM)\}$ be the tangent and cotangent bundles of $M$ respectively.

A $C^0$-Finsler structure on $M$ is a continuous function $F:TM\rightarrow\mathbb{R}$ such that $F(x,\cdot):T_xM\rightarrow\mathbb{R}$ is an asymmetric norm for every $x\in M$ (see Definition \ref{definition-asymmetric-norm}).
Polyhedral Finsler structures (or $p$-Finsler structures) are $C^0$-Finsler structures such that its restriction to tangent spaces are asymmetric norms which closed unit balls are polyhedra with the origin in its interior.

We use the term Finsler structures for the smooth case (see \cite{BaoChernShen}). 
We can see, even in simple examples, that the behavior of geodesics observed in Riemannian and Finsler manifolds are not satisfied in $C^0$-Finsler manifolds. 
For instance, geodesics don't need to be smooth and given $p$ in $M$ and $v\in T_pM$, we can have multiple geodesics $\gamma:(-\varepsilon, \varepsilon) \rightarrow M$ satisfying $\gamma(0)=p$ and $\gamma^\prime (0) = v$ or else no geodesics at all (see \cite{Fukuoka-large-family}, \cite{Gribanova}).

Differential calculus is one of the main techniques to study Finsler manifolds, but it can not be applied directly on $C^0$-Finsler manifolds. 
We need other techniques as we explain in sequel.

Metric geometry is one of these techniques. It can be used, for instance, to prove the local existence of minimizing curves that connect two points (see \cite{Burago, Mennucci}).
It can also be used to calculate explicitly geodesics for some specific $C^0$-Finsler manifolds (see \cite{Fukuoka-large-family}).

Another approach that can be useful for the study of $C^0$-Finsler manifolds $(M,F)$ is to approximate $F$ by a one-parameter family of Finsler structures $F_\varepsilon$.
In \cite{Fukuoka-Setti}, the authors proved that we can approximate a $C^0$-Finsler structure $F$ by a one-parameter family  $F_\varepsilon$ of Finsler structures such that if $F$ is a Finsler structure, then the various connections of the Finsler geometry and the flag curvature of $F_\varepsilon$ converge uniformly in compact subsets for the respective objects of $F$ when $\varepsilon$ tends to zero.
The technique used there was convolution with mollifiers.

We can assume that a $C^0$-Finsler structure has some kind of horizontal differentiability and we say that these structures are of Pontryagin type (See Item 3 of Definition \ref{definicao Pontryagin type}).
What is gained with this additional condition is that the Pontryagin's Maximum Principle (PMP) of the optimal control theory can be applied and the geodesic field of Finsler Geometry (and consequently, of Riemannian Geometry) can be generalized: in \cite{Fukuoka-Rodrigues}, the authors define the extended geodesic field $\mathcal E$ of a Pontryagin type $C^0$-Finsler manifold, which is a multivalued mapping that associates each $(x,\xi) \in TM^\ast \backslash 0$ to a subset of $T_{(x,\xi)}(T^\ast M\backslash 0)$.
The integral curves $(x(t),\xi(t))$ of the differential inclusion $(x^\prime(t),\xi^\prime(t)) \in \mathcal E(x(t),\xi(t))$ are the Pontryagin extremals.
The projection $x(t)$ of the Pontryagin extremals $(x(t),\xi(t))$ are extremals of $(M,F)$ and minimizing curves parametrized by arclength on the domain of definition of the control vector fields are always extremals. 

The use of the PMP for the study of extremals in $C^0$-Finsler manifolds was available since its creation by Pontryagin and his students \cite{Pontryagin}, because it deals with variational problems on non-differentiable length structures.
But in practice, this type of approach took a long time to become popular.
One of the first works that proposes the use of PMP for a wide class of geometric problems is \cite{Agrachev}.

Now we present some works in the literature that are related to the subject of this paper.
This list is not exhaustive and some relevant but not directly related subjects (such as sub-Riemannian geometry) were left out.
For those interested in this topic, see \cite{Agrachev-Barilari-Boscain}.

The PMP is well suited for Lie groups $G$ endowed with left invariant $C^0$-Finsler structures because they are of Pontryagin type.
Actually, a more general result holds: Let $M$ be a differentiable manifold and $G$ a transitive group of transformations over $M$.
Suppose that a $C^0$-Finsler structure $F$ on $M$ is $G$-invariant.
Then $F$ is of Pontryagin type (see \cite{Fukuoka-Rodrigues}).
In this work, we denote the Lie algebra of $G$ by $\mathfrak g$.

Let $M$ be a differentiable manifold endowed with a completely non-holonomic distribution $\mathcal D$ with constant rank $d$.
A $C^0$-sub-Finsler structure $F$ is a correspondence such that every $x \in M$ is associated to an asymmetric norm defined on $\mathcal D(x)$ in a continuous way, that is, if $X$ is a continuous horizontal vector field on $M$, then $F(X): M \rightarrow \mathbb R$ is continuous.
In a connected Lie group endowed with a left invariant distribution, the left invariant $C^0$-sub-Finsler structures can be studied using the PMP.

In \cite{Gribanova}, the author classify all minimizing paths on the $2$-dimensional non-abelian, simply connected Lie group endowed with a left invariant symmetric $C^0$-Finsler structure $F$ (the restriction of $F$ to each tangent space is a norm).
She uses the PMP and solves explicitly the resulting equations.

In \cite{Ardentov-Donne-Sachkov}, the authors study left invariant sub-Finsler structures of rank 2 in Cartan groups.
They choose two left invariant vector fields $X_1$ and $X_2$ on $G$ that generate the distribution and consider the maximum norm with respect to the frame given by the basis $\{X_1, X_2\}$.
They demonstrate that extremals of singular type are minimizers and extremals of type "bang-bang" are geodesics.

In \cite{Ardentov-Loku}, the authors study sub-Finsler structures of rank 2 in unimodular Lie groups of dimension 3.
The technique used here is convex trigonometry and they also use the control system defined by two left invariant vector fields that generate the distribution.
They study several cases of $C^0$-sub-Finsler structures, for example, with the unit ball of $F$ being polyhedral, strictly convex, $L^p$ norm, etc.

Let $G$ be a $2$-step Carnot group, that is, a Lie group such that the Lie algebra can be decomposed as $\mathfrak g = V_1 \oplus V_2$, with $[V_1,V_1]=V_2$ and $[V_1,V_2]=0$.
Consider a left invariant $C^0$-sub-Finsler structure $F$ defined on the distribution corresponding to $V_1$.
In \cite{Hakavuori}, the author proves that if $F$ is strictly convex (the restriction of $F$ to each tangent space is strictly convex), then isometries from $\mathbb R $ to $G$ are affine applications, that is, they are compositions of group homomorphism with left translations.

Many recent works use PMP on Lie groups with left invariant $C^0$-sub-Finsler structures of rank 2.
Among them we can cite \cite{ Ardentov-Loku, Bere-Zubareva, Loku, Sachkov}.
In these works, the authors calculate  extremals explicitly and study their properties.

Now we present this work.

In Section \ref{preliminaries} we fix notations and present the prerequisites that are necessary for the development of this work.

Let $G$ be a Lie group endowed with a left invariant $C^0$-Finsler structure $F$.
We consider the control system given by the left invariant unit vector fields.
The controls are measurable functions $u(t)$ defined on the unit sphere $S_{F_e}$ of $(\mathfrak g, F_e)$, where $F_e$ is the restriction of $F$ to $\mathfrak g$.
In Section \ref{The Extended Geodesic Field of a Lie Group}, we obtain the Lie algebra equivalent of the differentiable inclusion $(x^\prime(t),\xi^\prime(t)) \in \mathcal E(x(t),\xi(t)))$, which is a Euler-Arnold type equation $\mathfrak{a}'(t)=-\textrm{ad}^*(u(t))(\mathfrak{a}(t))$, where $\mathfrak a(t)$ is an absolutely continuous curve on the dual space $\mathfrak g^\ast$ of $\mathfrak g$ and the measurable control $u(t)$ maximizes $\mathfrak a(t)$ on $S_{F_e}$ (for Euler-Arnold equation, see \cite{Arnold-Khesin}).
A solution $(u(t),\mathfrak a(t))$ of $\mathfrak{a}'(t)=-\textrm{ad}^*(u(t))(\mathfrak{a}(t))$ is called a Pontryagin extremal of $(G,F)$ and $\mathfrak a(t)$ is its vertical part.

In Section \ref{secao curvatura}, we introduce the limit flag curvature for Lie groups endowed with a left invariant $p$-Finsler structure. 
Let $\mathfrak{a}\in \mathfrak{g}^\ast\backslash \{0\}$ be a functional and $v_1$ be a point that maximizes $\mathfrak{a}$ in $S_{F_e}$. 
Notice that if $L$ is a face of $S_{F_e}$ containing $v_1$, then $L \subset \{v_1\} + \ker \mathfrak{a}$.
We are interested in defining a curvature of a plane $\spann \{v_1,v_2\}$ with respect to $\mathfrak{a}$, where $v_2$ is chosen in $\ker \mathfrak{a}$ for convenience. 
We consider a family of left invariant Riemannian metrics $\{g_k\}_{k \in \mathbb N}$ on $G$ such that the respective sequence of unit spheres $S_{g_k}\subset (\mathfrak g,g_k\vert_{\mathfrak g})$ centered at the origin contains $v_1$ and converges locally to $\{v_1\}+\ker \mathfrak a$ in a neighborhood of $v_1$ when $k \rightarrow \infty$.
The sequence of sectional curvatures of $\spann \{v_1,v_2\} \subset (\mathfrak g,g_k\vert_{g})$ is a polynomial in $k$ and we denote the coefficient of highest degree of this polynomial by $\mathcal K_{\mathcal{B}}(\mathfrak{a})$ (the notation $\mathcal B$ and technical details are explained in Section \ref{secao curvatura}). 

Let $F_\ast$ be the dual norm of $F_{e}$ on $\mathfrak g^\ast$.
$F_\ast$ is also a polyhedral asymmetric norm.
Due to the PMP, the vertical part of a Pontryagin extremal lies on a sphere $S_{F_\ast}[0,r] \subset \mathfrak g^\ast$. 
In Section \ref{uniqueness} we study Lie groups endowed with a left invariant $p$-Finsler structure.
We prove that the equation $\mathcal K_{\mathcal B}(\mathfrak{a})=0 \; (\neq 0)$ depends only on $\mathfrak a$ and $v_2$ and we denote it by $\mathcal{K}(\mathfrak{a},v_2)=0 \; (\neq0)$ (see Proposition \ref{proposicao k fresco e algebrico} and Remark \ref{observacao depende so de v2}).
If $G$ is a three-dimensional Lie group, we prove that $\mathcal{K}(\mathfrak{a},v_2)$ doesn't depend on $v_2$ (see Proposition \ref{proposition ka zero kera subalgebra} and Remark \ref{caso dimensao 3 k fresco nao depende v2}).
Theorems \ref{multiplas solucoes 1}, \ref{multiplas solucoes 2}, \ref{teorema principal} and \ref{teorema principal 2} are the main results of this work and they relate two properties of the vertical part $\mathfrak{a}(t)$ of a Pontryagin extremal:  $\mathfrak{a}(t)$ admitting infintely many controls $\tilde u(t)$ such that $(\tilde u(t), \mathfrak{a}(t))$ is a Pontryagin extremal and the set $\mathcal{I} = \{t\in I; \text{ there exist }v_2 \in \ker \mathfrak{a}(t)\backslash \{0\}\text{ such that } \mathcal{K}(\mathfrak{a}(t),v_2)=0\}$ having positive measure.

Finally Section \ref{final-remarks} is devoted to final remarks and suggestions for future works.

In Finsler (and Riemannian) Geometry, geodesics and curvature are related by Jacobi fields.
The main contribution of this work is to show that an asymptotic curvature can be defined on $\mathfrak g^\ast$ and that it can be related with the behaviour of extremals in $p$-Finsler manifolds.

This work was done during the Ph.D. of the first author under the supervision of the second author at State University of Maringá, Brazil.

\section{Preliminaries}
\label{preliminaries}

In this section we fix some notations and present some definitions that are used in this work. 
For asymmetric norms, compare with \cite{Cobzas}. 
For convex polyhedral sets, see \cite{Grunbaum}.
In Theorem \ref{existencia Finsler poliedral suave horizontal} we prove the existence of $p$-Finsler structures on a smooth manifold. 
It is an immediate consequence of the theory presented here.  
Meanwhile, we also present some subjects of $C^0$-Finsler geometry, convex geometry and duality between asymmetric norms that will be useful afterward.

Let $V$ be a finite dimensional real vector space and $V^\ast$ its dual space.
We consider the usual topology on $V$ and $V^\ast$.
If $X$ is a subset of a topological space, its closure, interior and boundary will be denoted by $\bar X$, $\inter X$ and $\partial X$ respectively.

\begin{definition}
	\label{definition-asymmetric-norm}
	An {\em asymmetric norm} on $V$ is a nonnegative function $F:V \rightarrow \mathbb{R}$ such that
	\begin{enumerate}
		\item $F(y)=0$ iff $y=0$;
		\item $F(\alpha y)=\alpha F(y)$ for every $\alpha > 0$ and $y\in V$ (positive homogeneity);
		\item $F(y + w) \leq F(y) + F(w)$ for every $y,w\in V$.
	\end{enumerate}
	The closed ball centered at $y \in (V,F)$ with radius $r\geq 0$ is defined as
	\[
	B_F[y,r]=\{w\in V;F(w-y) \leq r\}
	\]
	and analogous notations hold for open balls $B_F(y,r)$ and spheres $S_F[y,r]$.
	The spheres centered at the origin with radius $1$ will be simply denoted by $S_F$.
\end{definition}

In particular, a norm is an asymmetric norm.

The proof of the next proposition is relatively simple, so we omit it.
\begin{proposition}
	\label{poliedral-bilipschitz}
	Let $F_1$ and $F_2$ be asymmetric norms on a finite dimensional vector space $V$ over $\mathbb R$.
	Then there exist constants $c,C>0$ such that 
	\[
	c.F_1(y) \leq F_2(y) \leq C.F_1(y)
	\]
	for every $y \in V$. Moreover, if $F$ is an asymmetric norm, then $F$ is continuous.
\end{proposition}

\begin{definition}
\label{definicao C0 Finsler structure}
A $C^0$-Finsler structure $F$ on a differentiable manifold $M$ is a continuous function $F: TM \rightarrow \mathbb R$ such that its restriction to each tangent space is an asymmetric norm. 
A differentiable manifold endowed with a $C^0$-Finsler structure is a $C^0$-Finsler manifold.
\end{definition}

Now we present some fundamental theory of $C^0$-Finsler manifolds.

\begin{proposition}
\label{proposicao estruturas equivalentes}
Let $F_1$ and $F_2$ be two $C^0$-Finsler structures on $M$. 
Then every $p\in M$ admit a neighborhood $U$ and $c,C>0$ such that 
\begin{equation}
\label{equacao normas comparaveis em TM}
c.F_1(x,y) \leq F_2(x,y) \leq C.F_1(x,y)
\end{equation}
for every $(x,y) \in TU$. 
\end{proposition}

\begin{proof}
Fix a Riemannian metric $g$ on $M$ and let $\mathcal{S}M$ be the sub-bundle of unit vectors of $TM$ with respect to $g$.
If we choose a neighborhood $U$ of $p$ with compact closure in $M$ and set
\[
C=\sup_{(x,y)\in \mathcal{S}U}\frac{F_2(x,y)}{F_1(x,y)} \;\;\;\text{ and }\;\;\; c = \inf_{(x,y)\in \mathcal{S}U}\frac{F_2(x,y)}{F_1(x,y)},
\]
then (\ref{equacao normas comparaveis em TM}) holds.
\end{proof}

\begin{definition}
Let $\gamma: [a,b] \rightarrow (M,F)$ be a path which is continuously differentiable by parts. 
The length of $\gamma$ with respect to $F$ is defined by
\[
\ell_F(\gamma) = \int_a^b F(\gamma(t),\gamma^\prime(t)).dt.
\]
We say that $\gamma$ is minimizing if $\ell_F(\gamma)\leq \ell_F(\eta)$ for every continuously differentiable by parts $\eta$ that connects $\gamma(a)$ and $\gamma(b)$. 
We say that $\gamma$ is a geodesic if $\gamma$ is locally minimizing, that is, if for every $t_0 \in [a,b]$, there exists a neighborhood $(t_0-\varepsilon,t_0+\varepsilon) \cap [a,b]$ of $t_0$ in $[a,b]$ such that $\gamma\vert_{I}$ is minimizing for every closed interval $I \subset (t_0-\varepsilon,t_0+\varepsilon) \cap [a,b]$.
\end{definition}

\begin{remark}
If $\tilde{\gamma}$ is the reverse curve of $\gamma$, then in general $\ell_F(\tilde{\gamma}) = \ell_F(\gamma)$ doesn't hold. 
Therefore $d_F: M \times M \rightarrow \mathbb{R}$ defined by 
\[
d_F(p,q) = \inf_{\gamma \in \mathcal{C}^1_{p,q}}\ell_F (\gamma),
\]
where $\mathcal{C}^1_{p,q}$ is the family of paths that are continuously differentiable by parts and connects $p$ and $q$, isn't a metric in general.
\end{remark}

\begin{definition}
\label{definicao bola aberta}
The open ball in $(M,F)$ with center $p$ and radius $r$ is defined by $B_F(p,r) = \{q \in M; d_F(p,q)< r\}$.
The closed ball $B_F[p,r]$ and the sphere $S_F[p,r]$ are defined analogously.
\end{definition}

\begin{remark}
\label{observacao metrica assimetrica}
If we choose a Riemannian metric $F_2=g$ in 
(\ref{equacao normas comparaveis em TM}) and consider that the family of open balls in a Riemannian manifold $(M,g)$ is a basis of the topology of the differentiable manifold $M$, then we can conclude that the family of open balls in $(M,F_1)$ is also a basis of the topology of $M$.
In particular, the connected components of $M$ coincides with the connected components of $(M,F)$.
In particular, $d_F(p,q) < \infty$ iff $p$ and $q$ lies in the same connected component of $M$.

Finally if $d_F$ is restricted to a connected component $\tilde M$ of $(M,F)$, then $d_F\vert_{\tilde{M}}$ is an {\em asymmetric metric}, that is, a mathematical object that satisfies every condition of a (finite) metric, except the symmetry.
\end{remark}

\begin{definition}
	Let $V$ be a finite dimensional vector space and $Z$ be a convex subset of $V$.
	The {\em affine hull} $\aff Z$ of $Z$ is the smallest affine subset of $V$ that contains $Z$.
	The {\em relative interior} $\ri Z$ of $Z$ is the interior of $Z$ considered as a subspace of $\aff Z$.  
\end{definition}

Now we return to the theory of polyhedral $C^0$-Finsler structures.
We begin with the theory on polyhedral asymmetric norms on finite dimensional real vector spaces $V$.
The definitions and results given here aren't in the most general setting.
We adapted the theory presented in \cite{Grunbaum} according to our necessities.

\begin{definition}
	A non-empty subset of $K\subset V$ is {\em polyhedral} if it is the intersection of a finite family of closed half spaces of $V$.
\end{definition}

\begin{definition}
	A polyhedral asymmetric norm $F:V \rightarrow \mathbb R$ on a vector space $V$ is an asymmetric norm such that its closed unit ball $B_F[0,1]$ is a polyhedral subset of $V$. 
\end{definition}

\begin{definition}
\label{definicao face}
A {\em supporting hyperplane} of a non-empty compact subset $A$ of $V$ is a hyperplane $H$ of $V$ such that $A$ is contained in one of the closed half spaces determined by $H$ and $A\cap H \neq \emptyset$. A subset $L$ of a polyhedral set $K\subset V$ is a face of $K$ if either $L=\emptyset$, $L=K$ or else if there exists a supporting hyperplane $H$ of $K$ such that $L=K\cap H$. A $0$-dimensional and a $1$-dimensional face of $K$ is also called a vertex and an edge of $K$, respectively.
\end{definition}

The polyhedral subsets we are interested in are closed balls $B_F[0;r]$ in $(V,F)$, which are bounded (compact) polyhedral subsets of $V$ with the origin in its interior.

\begin{definition}
A maximal proper face of a polyhedral subset $B$ of $V$ is a {\em facet} of $B$. 
\end{definition}

The next proposition gives a detailed information about the configuration of the faces of a compact polyhedral subset of $V$ with non-empty interior (see Sections 2.6 and 3.1 of \cite{Grunbaum}). 
In order to follow Proposition \ref{proposicao organizacao faces} from this reference, some remarks are useful.
\begin{itemize}
\item Polytope is equivalent to bounded polyhedral subset;
\item A subset of a polyhedral subset $B$ of $V$ is a poonem of $B$ iff it is a face of $B$; 
\end{itemize}

\begin{proposition}
\label{proposicao organizacao faces}
Let $V$ be an $n$-dimensional real vector space and $B$ be a bounded polyhedral subset of $V$ with non-empty interior. Then
\begin{itemize}
\item There exists the smallest 
family of closed half spaces $\{H_1^-, \ldots, H_l^- \}$ of $V$ such that $B=\cap_{i=1}^l H_i^-$. 
The affine hulls of the facets of $B$ are $H_i := \partial H_i^-$.
In particular, the dimension of the facets is $(n-1)$.
\item Each facet of a facet of $B$ is the intersection of two facets of $B$;
\item $\partial B$ is the union of facets of $K$;
\item Every non-empty proper face of $B$ is the the intersection of facets of $B$. 
\end{itemize}
\end{proposition}

Now we proceed studying polyhedral asymmetric norms.

\begin{remark}
	\label{funcional e semi espaco}
	The map $\alpha\in V^*\backslash \{0\}\mapsto \alpha^{-1}(-\infty,1]$ is a bijection between $V^\ast \backslash \{0\}$ and the family of closed half spaces that contains the origin in its interior.
\end{remark}

\begin{proposition}
	\label{F em termos de funcionais}
	Let $F$ be a polyhedral asymmetric norm. Suppose that
	$$	B_F[0,1]=\displaystyle{\bigcap_{i=1}^m H_i^-,}$$
	where $H_i^-$ are closed half spaces.
	Then 
	\begin{equation}
		\label{F em termos de alpha}
		F=\max\{\alpha_1, \ldots, \alpha_m\},
	\end{equation}
	where $\alpha_i\in V^{\ast}\setminus\{0\}$ and $H_i^-=\alpha_i^{-1}(-\infty,1]$ for $i\in\{1, \ldots, m\}$.
\end{proposition}

\begin{proof}
	
	From Proposition \ref{poliedral-bilipschitz}, $0\in \inter H_i^-$ for every $i\in\{1, \ldots, m\}$ and there exist a unique linear functional $\alpha_i:V \rightarrow \mathbb R$ such that $\alpha_i^{-1}(-\infty,1] = H_i^-$ due to Remark \ref{funcional e semi espaco}. Now notice that both sides of Equation (\ref{F em termos de alpha}) are positively homogeneous. 
	Therefore it is enough to prove it on $S_F$.
	First of all $\alpha_i(y) \leq F(y)$ for every $y\in S_F$ because $S_F \subset H_i^-$.
	In order to prove that $F(y) \leq \max\{\alpha_i, \ldots, \alpha_m\}$, let $y \in S_F$.
	Then $y \in H_i:=\partial H_i^-$ for some $i\in \{1, \ldots, m\}$ and we have that $F(y) = \alpha_i(y) = 1 \leq \max \{\alpha_1 (y), \ldots, \alpha_m(y)\}$. 
\end{proof}

The next proposition is the reciprocal of Proposition \ref{F em termos de funcionais}.

\begin{proposition}\label{alpha gera Finsler}
	Let $\{\alpha_1, \ldots, \alpha_m\} \subset V^\ast\backslash \{0\}$ be a family of functionals such that for every $y\in V\backslash \{0\}$, there exist $i\in \{1, \ldots, m\}$ satisfying $\alpha_i(y) >0$. Then $\max\{\alpha_1, \ldots, \alpha_m\}$ is a polyhedral asymmetric norm.
\end{proposition}

\begin{proof}
	It is straightforward that Items (1) and (2) of Definition \ref{definition-asymmetric-norm} holds for $\max \{\alpha_1, \ldots, \alpha_m\}$.
	For the triangular inequality, if $y,w \in V$, then there exist $j\in \{1, \ldots, m\}$ such that
	\[
	\max_{i=1, \ldots, m}\alpha_i (y+w)=\alpha_j (y+w)
	\]
	and
	\[
	\max_{i=1, \ldots, m}\alpha_i (y+w) = \alpha_j(y) + \alpha_j(w) \leq \max_{i=1, \ldots, m}\alpha_i (y) + \max_{i=1, \ldots, m}\alpha_i (w),
	\]
	what settles the proposition.
\end{proof}

\begin{lemma}\label{Sum of polyhedral}
	Let $F_1$ and $F_2$ be polyhedral asymmetric norms and $a_1, a_2 > 0$. 
	Then $a_1F_1 + a_2F_2$, is a polyhedral asymmetric norm.
\end{lemma}

\begin{proof}
	It is enough to prove that if $F_1$ and $F_2$ are polyhedral asymmetric norms and $a>0$, then $F_1 + F_2$ and $a.F_1$ are polyhedral asymmetric norms.
	The latter statement is trivial.
	For the proof of the former statement, set $F_1= \max\limits_{i=1,\ldots,k} \alpha_i$ and $F_2=\max\limits_{j=1,\ldots, m} \beta_j$, where $\{\alpha_i\}_{i=1,\ldots, k}$ and $\{\beta_j\}_{j = 1, \ldots, m}$ are families of linear functionals satisfying the conditions of Proposition \ref{alpha gera Finsler}. 
	Notice that
	\[
	F_1+F_2=\max_{i=1,\ldots,k} \{ \alpha_i \} + \max_{j=1,\ldots,m} \{ \beta_j \}=\max_{\substack{i=1,\ldots,k \\ j=1,\ldots,m}} \{ \alpha_i+\beta_j \}
	\]
	and $F_1+F_2$ is a polyhedral asymmetric norm due to Proposition \ref{alpha gera Finsler}.
\end{proof}

\begin{definition}
\label{definicao estrutura p-Finsler}
	Let $M$ be a differentiable manifold. A polyhedral Finsler structure (or shortly, $p$-Finsler structure) $F$ on $M$ is a $C^0$-Finsler structure such that $F_x:=F(x,\cdot):T_xM \rightarrow \mathbb R$ is a polyhedral asymmetric norm for every $x \in M$.
	\end{definition}

\begin{proposition}
	\label{existencia Finsler poliedral suave horizontal}
	Every differentiable manifold $M$ admits a $p$-Finsler structure.
\end{proposition}

\begin{proof}
	Let $\{(U_i,\mathbf x)\}_{i \in \mathbb N}$ be a locally finite cover by coordinate neighborhoods on $M$. 
	Let $\{\eta_i: M \rightarrow \mathbb R\}_{i \in \mathbb N}$ be a smooth partition of unity subordinated to $\{(U_i,\mathbf x)\}_{i \in \mathbb N}$.
	Consider these coordinate neighborhoods endowed with constant $p$-Finsler structures $F_i(x,y)=F_i(y)$ and set $F=\sum_{i \in \mathbb N} \eta_i F_i$.
	Then this sum is locally finite and $F$ is a $p$-Finsler structure on $M$ due to Lemma \ref{Sum of polyhedral}.
\end{proof}

\begin{definition}
	A differentiable manifold $M$ endowed with a $p$-Finsler structure is called a $p$-Finsler manifold.
\end{definition}

We end this section with some results about an asymmetric norm $F$ and its dual asymmetric norm $F_\ast$.
The basic reference here is Section 3.4 of \cite{Grunbaum}.
It is necessary to remark that in this work, the author study finite dimensional vector space $V$ endowed with an inner product $\left< \cdot, \cdot\right>$, and several types of duality are between objects of $(V,\left< \cdot, \cdot \right>)$.
But the inner product induces an isomorphism $v \mapsto \left< v, \cdot\right>$ between $V$ and $V^\ast$, and this isomorphism induces the corresponding duality between objects of $V$ and $V^\ast$ that will be used in this work.

\begin{remark}
\label{obsevacao correspondencia reversa}
Let $F_\ast : V^\ast \rightarrow \mathbb R$ be the dual asymmetric norm 
\[
F_\ast (\mathfrak{a}) = \max_{v\in B_F[0,1]} \mathfrak{a}(v)
\] 
of $F$. 
Then the unit ball $B_{F_\ast}[0,1] \subset V^\ast$ is the polar set of $B_F[0,1]$. 
Therefore $B_{F_{\ast}}[0,1]$ is a polyhedral subset of $V^\ast$ and consequently $F_\ast$ is a polyhedral asymmetric norm (see Exercise 5 (viii), Section 3.4 of \cite{Grunbaum}). 
The inclusion reversing correspondence $\Psi$ between the faces of $S_F$ and $S_{F_\ast}$ is given as follows:
If $L$ is a $k$-dimensional face of $S_F$, then $L^\ast = \Psi(L) = \{\mathfrak{a} \in S_{F_\ast}; \mathfrak{a}(v) = 1 \text{ for every }v \in L\}$ is a $(n-k-1)$-dimensional face of $S_{F_\ast}$. 
Similarly $\Psi^{-1}$ is given by $L= \Psi^{-1}(L^\ast) = \{v \in S_F;\mathfrak{a}(v) =1 \text{ for every }\mathfrak{a} \in L^\ast\}$ (see \cite[Section 3.4]{Grunbaum}). 
\end{remark}

\begin{remark}
	\label{remark face que maximiza funcional}
	Given $\mathfrak{a} \in \ri L^\ast$, we claim that $\mathfrak{a}$ is maximized in $S_F$ by $L = \Psi^{-1}(L^\ast)$.
	In fact, notice that $\mathfrak{a}(v) = 1$ for every $v \in L$ due to $L= \{v \in S_F;\mathfrak{a}(v) =1 \text{ for every }\mathfrak{a} \in L^\ast\}$. But we can't have $\mathfrak{a}(v)=1$ for $v \not \in L$ otherwise $\mathfrak{a}$ would be maximized by a face properly containing $L$, which would imply that $\mathfrak{a}$ lies in a face of $S_{F_\ast}$ properly contained in $L^\ast$, contradicting $\mathfrak{a} \in \ri L^\ast$.
\end{remark}

\section{The Extended Geodesic Field}
\label{The Extended Geodesic Field of a Lie Group}

In \cite{Fukuoka-Rodrigues} the authors define the Pontryagin type
$C^0$-Finsler structures which are structures that satisfy the minimum requirements of Pontryagin Maximum Principle. 

\begin{definition}
\label{definicao Pontryagin type}
A $C^0$-Finsler manifold $(M,F)$ is of Pontryagin type at $p\in M$ if there exist a neighborhood $U$ of $p$, a coordinate system $\phi = (x^1,\cdots,x^n):U\rightarrow\mathbb{R}^n$, with the respective natural coordinate system $d\phi = (x^1,\cdots,x^n,y^1,\cdots,y^n):TU\rightarrow\mathbb{R}^{2n}$ on the tangent bundle, and a family of $C^1$ unit vector fields $$\{x\mapsto X_u(x)=(y^1(x^1,\cdots,x^n,u),\cdots,y^n(x^1,\cdots,x^n,u)); u\in S^{n-1}\}$$ on $U$ parametrized by $u\in S^{n-1}$ such that 
\begin{itemize}
	\item[1.] $u\mapsto X_u(x)$ is a homeomorphism from $S^{n-1}\subset \mathbb{R}^n$ onto the unit sphere of $(T_xM,F(x,\cdot))$ for every $x\in U$;
	\item[2.] $(x,u)\mapsto(y^1(x^1,\cdots,x^n,u),\cdots, y^n(x^1,\cdots,x^n,u))$ is continuous;
	\item[3.] $(x,u)\mapsto\left(\frac{\partial y^1}{\partial x^i}(x^1,\cdots,x^n,u), \cdots,\frac{\partial y^n}{\partial x^i}(x^1,\cdots,x^n,u)\right)$ is continuous for every $i\in\{1,\cdots,n\}$.
\end{itemize}
We say that $F$ is of Pontryagin type on $M$ if it is of Pontryagin type for every $p\in M$.
\end{definition}

\begin{remark}
\label{observacao tipo Pontryagin nao depende sistema de coordenadas}
This definition doesn't depend on the choice of the coordinate system (see \cite[Remark 3.2]{Fukuoka-Rodrigues}).
\end{remark}

Denote $X_u(x)=f^i(x,u)\frac{\partial}{\partial x^i}=(f^1(x,u),\cdots,f^n(x,u))$ ({\em From now on the Einstein summation convention is in place}). 
The extended geodesic field is the rule $\mathcal{E}$ that associates each $(x,\xi)\in T^*U\backslash 0$ to the set \begin{eqnarray}\label{campo geodesico estendido}\mathcal{E}(x,\xi)=\left\{f^i(x,u)\frac{\partial}{\partial x^i}-\xi_j\frac{\partial f^j}{\partial x^i}(x,u)\frac{\partial}{\partial \xi_i}; \ u\in \mathcal{C}(x,\xi)\right\},
\end{eqnarray}
where $\mathcal{C}(x,\xi)=\{u\in S^{n-1}; \ \xi(X_u(x))=\max_{v\in S^{n-1}}\xi(X_v(x))\}$. 
An absolutely continuous curve $\gamma:[a,b]\rightarrow T^*M$ such that $\gamma'(t)\in\mathcal{E}(\gamma(t))$ for almost every $t\in [a,b]$ is called a Pontryagin extremal of $(M,F)$. \index{{Pontryagin}! extremal}

As a consequence of Pontryagin Maximum Principle, if $x(t)$ is a length minimizer on $U$, then there exist a Pontryagin extremal $(x(t),\xi(t))$ of $(U ,F)$ and 
\begin{equation}
\label{equacao energia constante pontryagin}
\mathcal{M}(x(t),\xi(t))=\max_{u\in S^{n-1}}\xi(t)(X_u(x(t))) = F_\ast (x(t),\xi(t))
\end{equation}
is constant.

Let $G$ be a Lie group with identity element $e$.
Consider the left translation $L_x: G \rightarrow G$, $L_x(z)=x.z$.
Let $F:TG \rightarrow \mathbb R$ be a left invariant $C^0$-Finsler structure $F$, that is, 
\[
F(x.z,(dL_x)_z(y)) = F(z,y)
\]
for every $x,z \in G$ and $y \in T_zG$.
In \cite[Section 8]{Fukuoka-Rodrigues} the authors prove that if $\varphi: \tilde G \times M \rightarrow M$ is a transitive differentiable action of a Lie group $\tilde G$ on a differentiable manifold $M$ and $F$ is a $\tilde G$-invariant $C^0$-Finsler structure on $M$, then $F$ is of Pontryagin type.
In particular, Lie groups endowed with left invariant $C^0$-Finsler structures are of Pontryagin type. 

From now on, $G$ is a Lie group endowed with a left invariant $C^0$-Finsler structure $F$. 
We denote its Lie algebra and its dual by $\mathfrak g$ and $\mathfrak g^\ast$ respectively.
The aim of this section is to  represent the extended geodesic field of $(G,F)$ on $\mathfrak{g}^*$.
Here we replace the control set $S^{n-1}$ by the unit sphere $\sfe \subset \mathfrak{g}$ because it is more convenient than the Euclidean sphere $S^{n-1}$.

In \cite{Sachkov-Agrachev}, the authors consider a Hamiltonian $h$ defined on $T^\ast G$ and they calculate the Hamiltonian system associated to $h$.
They also consider the case where $h$ is left invariant and they calculate the reduction of the Hamiltonian system to $\mathfrak g^\ast$.
These calculations can be adapted for our case where a control is also considered.
But we will give the complete proof of the representation of the extended geodesic field of $(G,F)$ on $\mathfrak g^\ast$ because it is more helpful for the reader than to point out the necessary adaptations which are necessary to prove this result in a format that we need.

The left invariant vector fields $x \mapsto X_u(x)=d(L_x)_e (u)$ define a family of unit vector fields on $G$ parametrized by $u$ satisfying all conditions of the definition of Pontryagin type. 
Choose a basis $\{e_1,\cdots,e_n\}$ of $\mathfrak{g}$ and let $(x^1,\cdots,x^n)$ be a coordinate system in a neighborhood $U$ of $e$ such that $\frac{\partial}{\partial x^i}(e) = e_i$, $i\in\{1,\cdots,n\}$.
For every $i\in\{1,\cdots,n\}$, let $X_i$ be the left invariant vector field such that $X_i(e)=e_i$. We can write $$X_i(x)=d(L_x)_e(e_i)=b_i^j(x) \frac{\partial}{\partial x^j} \in TU,$$ where $b_i^j\in C^{\infty}(U)$. Since $X_u$ is left invariant, 
$$f^i(x,u)\frac{\partial}{\partial x^i}=X_u(x)=d(L_x)_e(u)=d(L_x)_e(u^i e_i)= u^i b_i^j(x)\frac{\partial}{\partial x^j},$$
implying that
$f^i(x,u)=u^j b_j^i(x).$

Therefore, in a Lie group, the extended geodesic field (\ref{campo geodesico estendido}) turns out to be \begin{eqnarray}\label{campo geodesico no grupo de Lie}\mathcal{E}(x,\xi)=\left\{u^j b_j^i(x)\frac{\partial}{\partial x^i}-\xi_j u^k\frac{\partial b_k^j}{\partial x^i}(x)\frac{\partial}{\partial \xi_i}; \ u\in \mathcal{C}(x,\xi)\right\}.
\end{eqnarray}

Now let us prove that $\mathcal{E}$ can be represented on $\mathfrak{g}^*$. In what follows, $\textrm{ad}^*$ is the infinitesimal coadjoint representation on $\mathfrak{g}^*$ (see \cite{SanMartin}).

\begin{theorem}
\label{teorema equivalencia fibrado cotagente algebras}
	Let $(G,F)$ be a $C^0$-Finsler Lie group with Lie algebra  $\mathfrak{g}$ and its dual $\mathfrak{g}^*$. 
	Let  $\widetilde{\mathcal{E}}$ be the rule that associates each element of $\mathfrak{g}^\ast\backslash \{0\}$ to the set \begin{eqnarray}\label{camponaalgebra}\widetilde{\mathcal{E}}(\mathfrak{a})=\{-\ad^*(u)(\mathfrak{a}); \ u\in \mathcal{C}(\mathfrak{a})\},
	\end{eqnarray} 
where $\mathcal{C}(\mathfrak{a}) := \{u\in \sfe; \ \mathfrak{a}(u)=\max_{v\in \sfe}\mathfrak{a}(v)\}$.
	Then finding Pontryagin extremals of $(G,F)$ is equivalent to  find $(u(t), \mathfrak{a}(t))$ in $S_{F_e}\times \mathfrak{g}^* \backslash\{0\}$ such that $\mathfrak{a}(t)$ is a solution of the differential inclusion \begin{eqnarray}\label{differential inclusion algebra}
		\mathfrak{a}'(t) \in \widetilde{\mathcal{E}}(\mathfrak{a}(t)),
	\end{eqnarray} 
	and $u(t)\in \mathcal{C}(\mathfrak{a}(t))$ is a measurable control.
\end{theorem}

\begin{proof} 
First of all we prove that finding Pontryagin extremals $(x(t),\xi (t))$ of $(G,F)$ is equivalent to find $(u(t), \mathfrak{a}(t))$ in $S_{F_e} \times \mathfrak{g}^\ast\backslash \{0\}$ such that $u(t)$ is a measurable control, $\mathfrak{a}(t)$ is an absolutely continuous function, $\mathfrak{a}^\prime(t)=-\mathrm{ad}^*(u(t))(\mathfrak{a}(t))$ and $u(t) \in \mathcal{C}(\mathfrak{a}(t))$. 

Let $(x(t),\xi(t))$ a Pontryagin extremal of $(G,F)$ and $u(t)\in\mathcal{C}(x(t),\xi(t))$ the corresponding admissible control. Then,
\begin{eqnarray}\label{sistema no grupo}\left\{\begin{array}{l}
		\frac{dx^i}{dt}= u^j(t) b_j^i(x(t))\\
		\frac{d\xi_i}{dt}=-\xi_j(t) u^k(t) \frac{\partial b_k^j}{\partial x^i}(x(t))
	\end{array}\right.\end{eqnarray}
Using the pullback of $L_{x(t)}$ at $e$, we define
\begin{eqnarray}\label{def de a} \mathfrak{a}(t)=d(L_{x(t)}^*)_e(\xi(t))=\xi(t)(d(L_{x(t)})_e)\in\mathfrak{g}^*.\end{eqnarray} 
Notice that given $y\in\mathfrak{g}$,
we have 
$$d(L_x)_e(y)=d(L_x)_e\left(y^ie_i\right)=y^id(L_x)_e(e_i)=y^ib_i^j(x)\frac{\partial}{\partial x^j}.$$
Then,
$$\begin{array}{rcl}\mathfrak{a}(t)(y)&=&d(L_{x(t)}^*)_e(\xi(t))(y)\\
	&=&\xi(t)(d(L_{x(t)})_e(y))\\
	&=&\xi_k(t)dx^k\left(y^ib_i^j(x(t))\frac{\partial}{\partial x^j}\right)\\
	&=&\xi_j(t)b_i^j(x(t))y^i,
\end{array}$$
and writing $\mathfrak{a}(t)=a_i(t)e^i$, where $\{e^1,\cdots,e^n\}$ is the dual basis of $\{e_1,\cdots,e_n\}$, we have
\begin{eqnarray}\label{relacao lambda e xi}a_i(t)=\xi_j(t)b_i^j(x(t)).\end{eqnarray} Calculating the derivative of the last equality and using the expressions of $\xi_i'$ and $(x^i)'$ from (\ref{sistema no grupo}) we get 
\begin{align}
	a_i^\prime(t) & = \xi_l '(t)b_i^l(x(t))+\xi_j(t)(x^l)'(t)\frac{\partial b_i^j}{\partial x^l}(x(t)) \nonumber \\
	& = -\xi_j(t) u^k(t) \frac{\partial b_k^j}{\partial x^l}(x(t))b_i^l(x(t))+\xi_j(t) u^k(t) b_k^l(x(t))\frac{\partial b_i^j}{\partial x^l}(x(t)) \label{recupera xi} \\
	& = \xi(t)\left[u^k(t) b_k^l(x(t))\frac{\partial}{\partial x^l},b_i^j(x(t))\frac{\partial}{\partial x^j}\right] \nonumber \\
	& = \xi(t)[X_{u(t)}(x(t)),d(L_{x(t)})_e(e_i)] \nonumber \\
	& = \xi(t)\circ d(L_{x(t)})_e[u(t),e_i] \nonumber \\
	& = \mathfrak{a}(t)(\textrm{ad}(u(t))(e_i)) \nonumber \\
	& = -\textrm{ad}^*(u(t))(\mathfrak{a}(t))\left(e_i\right).\nonumber
\end{align}
Therefore
\begin{eqnarray}\label{sistema na algebra}
	\mathfrak{a}^\prime(t)=-\textrm{ad}^*(u(t))(\mathfrak{a}(t)),
\end{eqnarray}
where $u(t)\in\mathcal{C}(\mathfrak{a}(t)) := \{u \in \sfe; \mathfrak{a}(t)(u) = \max_{v \in \sfe} \mathfrak{a}(t)(v)\}$ is an admissible control and 
\begin{equation}
	\label{norma dual constante}
	\mathcal{M}(\mathfrak{a}(t))=\max_{u\in \sfe}\mathfrak{a}(t)(u)
\end{equation}
is constant.
Therefore a Pontryagin extremal $(x(t),\xi(t))$ induces a function $(u(t),\mathfrak a(t))$ which is essentially $(x^\prime(t), \xi(t))$ represented in $\mathfrak g \times \mathfrak g^\ast$ and satisfies (\ref{sistema na algebra}).

Reciprocally, fix a point $x_0\in G$ and a solution $(u(t),\mathfrak{a}(t))$ of (\ref{sistema na algebra}) such that $\mathfrak a(t)$ is absolutely continuous and $u(t)\in\mathcal{C}(\mathfrak{a}(t))$ is bounded and measurable. 
Then, by the Carathéodory existence and uniqueness theorem (see \cite{Hale}), there is a unique absolutely continuous curve $x(t)$ which is solution of 
\begin{equation}
\label{equacao recupera trajetoria}
\dot{x}=d(L_x)_e(u(t))=X_{u(t)}(x)
\end{equation}
and $x(0)=x_0$. 
The proof that this solution extends to the whole interval of definition $I$ of $(u(t),\mathfrak{a}(t))$ can be done adapting the proof of the classical Hopf-Rinow theorem of Riemannian geometry. 
Fix a left-invariant Riemannian metric $g$ on $G$. 
Then $(G,g)$ is a complete Riemannian manifold.
Notice that $c\ell_g (x) \leq \ell_F (x) \leq C.\ell_g (x)$ due to Proposition 1 and the left invariance of $F$ and $g$. 
If $(t_i)$ is a Cauchy sequence in $I$, then $(x (t_i))$ is a Cauchy sequence in $(G,g)$ because $x(t)$ is parametrized by arclength in $(G,F)$. 
Therefore any solution $x(t)$ of (\ref{equacao recupera trajetoria}) defined on an open interval $(a,b)$, with $b\in I$, can be extended continuously to $b$. 
If $b$ is an interior point of $I$, then the solution of (\ref{equacao recupera trajetoria}) can be extended beyond $b$.
Consequently $x(t)$ can be extended to $I$
and 
$$
\frac{dx^i}{dt}= u^j(t) b_j^i(x(t))
$$
holds in $I$.
 
Define $\xi(t)=d(L_{(x(t))^{-1}}^*)_{x(t)}(\mathfrak{a}(t))\in T_{x(t)}^*G$, which is equivalent to (\ref{def de a}) and implies (\ref{relacao lambda e xi}). Since the product of two absolutely continuous functions defined on a closed interval is absolutely continuous and $\mathfrak{a}(t)$ is absolutely continuous, so is $\xi$ due to (\ref{relacao lambda e xi}), the smoothness of $b^j_i$ and the fact that the matrix $(b^j_i(x))$ is everywhere invertible. 
Moreover, from (\ref{relacao lambda e xi}) and (\ref{sistema na algebra}), we get (\ref{recupera xi}) and
$$-\xi_j(t)u^k(t) \frac{\partial b_k^j}{\partial x^l}(x(t))b_i^l(x(t))+\xi_j(t)(b_i^j(x(t)))'=a_i'(t)=(\xi_j(t)b_i^j(x(t)))'.$$ 
Therefore 
$$
[(\xi_l(t))'+\xi_j(t)u^k(t) \frac{\partial b_k^j}{\partial x^l}(x(t))]b_i^l(x(t))=0
$$ 
and since the matrix $(b_i^l(x))$ is invertible, we have  $$\frac{d\,\xi_l(t)}{dt}=-\xi_j(t)f^k(e,u(t))\frac{\partial b_k^j}{\partial x^l}(x(t))$$ for every $l\in\{1,\cdots,n\}$. Thus $(x(t),\xi(t))$ is solution of (\ref{sistema no grupo}) for $u(t)\in\mathcal{C}(\mathfrak{a}(t))$.

Finally notice that finding $(u(t), \mathfrak{a}(t))$ in $S_{F_e}\times \mathfrak{g}^*\backslash\{0\}$ such that $\mathfrak{a}(t)$ is a solution of  (\ref{camponaalgebra}) and $u(t)\in \mathcal{C}(\mathfrak{a}(t))$ is a measurable control is equivalent to find $(u(t), \mathfrak{a}(t))$ satisfying (\ref{sistema na algebra}), what settles the proof.
\end{proof}

\begin{definition} 
A solution $(u(t),\mathfrak{a}(t))$ of $\mathfrak{a}^\prime(t)=-\textrm{ad}^*(u(t))(\mathfrak{a}(t))$, where $u(t) \in \mathcal{C}(a(t))$ is a measurable function, is also called a Pontryagin extremal of $(G,F)$ and $\mathfrak{a}(t)$ is called the vertical part of a Pontryagin extremal.
\end{definition}

\begin{remark}
\label{nao e suficiente ser solucao inclusao diferencial}
Notice that the solution $\mathfrak{a}(t)$ of the differential inclusion $\mathfrak{a}'(t) \in \widetilde{\mathcal{E}}(\mathfrak{a}(t))$ isn't enough to determine a Pontryagin extremal. 
We also need a control $u(t)$ such that $(u(t),\mathfrak{a}(t))$ is a Pontryagin extremal.
\end{remark}

\section{The Limit Flag Curvature}\label{secao curvatura}

In this section, we introduce the asymptotic expansion of the flag curvature for $p$-Finsler Lie groups with respect to $\mathfrak{a} \in \mathfrak{g}^\ast \backslash \{0\}$. This definition is given through approximation by sectional curvatures of a family of Riemannian metrics.

\begin{remark}
	\label{quase sempre}
	The term ``flag curvature" comes from Finsler geometry (see \cite{BaoChernShen}). It is a generalization of sectional curvature of Riemannian geometry. 
	Let us explain with more details.
	
	In Riemannian geometry, the Riemannian metric $g$ is determined by its components $g_{ij} = g(\partial/\partial x_i, \partial / \partial x_j)$ with respect to a coordinate system $\{x^1, \ldots, x^n\}$ on $M$.
	It is well known that the sectional curvatures are determined in terms of $g_{ij}$.
	 
	Let $\pi: TM \rightarrow M$ be the natural projection. 
	In Finsler geometry, the fundamental tensor $\tilde g$ is defined on the pulled back tangent bundle $\pi^\ast TM$.
	In the natural coordinate system $(x,y) = (x^1, \ldots , x^n,y^1, \ldots, y^n)$ of the slit tangent bundle $TM\backslash 0 := \{(x,y) \in TM, y\neq 0\}$, $\tilde g$ is given by
	\[
	\tilde g_{ij(x,y)} dx^i \otimes dx^j = \frac{1}{2}\frac{\partial F^2}{\partial y^i \partial y^j} dx^i \otimes dx^j \text{ (See \cite{BaoChernShen})}.
	\]
	Notice that in Finsler geometry, $\tilde g_{ij}$ also depends on $y$. 
	The key idea here is that $\tilde g_{ij(x,y)} dx^i \otimes dx^j$ in $(x,y)\in TM\backslash 0$ can be seen as the best Riemannian approximation of $F^2\vert_{T_xM}$ in a neighborhood of $y \in T_xM$.

	The flag curvature of a Finsler manifold is completely determined by its fundamental tensor, and it depends not only on the two-dimensional subspace $\sigma$ of $T_xM$, but also on the non-zero vector $y\in \sigma$.
	In other words, it depends on the flag $\{y,\sigma\}$.
	Therefore when we have $C^0$-Finsler structure $F$ such that the restriction of $F^2$ to their tangent spaces aren't inner products, it is natural to consider Riemannian approximations of $F^2\vert_{T_xM}$ in a neighborhood of $y$.
	We will use this idea in this section.
		
	Although the main idea to define limit flag curvature comes from Finsler geometry, the calculations made here are purely from Riemannian geometry since the approximation is made using Riemannian metrics. 
\end{remark} 

Let $G$ be a Lie group endowed with a left invariant Riemannian metric $g$. Then the sectional curvatures are calculated in \cite{Milnor} as follows. Fix an oriented orthonormal basis $\{v_1,\cdots , v_n\}$ on the Lie algebra $\mathfrak{g}$ and define the structure constants $\alpha^{g}_{ijk}$ by $[v_i, v_j] = \sum_k\alpha^{g}_{ijk}v_k$. Observe that $\alpha^{g}_{ijk} = \langle[v_i, v_j], v_k\rangle$. The sectional curvature of $\textnormal{span} \{v_1, v_2\}$ is given by
\begin{eqnarray}\kappa_g(v_1,v_2)&=&\sum_j\left(\frac{1}{2}\alpha^{g}_{12j}(-\alpha^{g}_{12j}+\alpha^{g}_{2j1}+\alpha^{g}_{j12})\right.\nonumber \\
	&-&\left.\frac{1}{4}(\alpha^{g}_{12j}-\alpha^{g}_{2j1}+\alpha^{g}_{j12})(\alpha^{g}_{12j}+\alpha^{g}_{2j1}-\alpha^{g}_{j12})-\alpha^{g}_{j11}\alpha^{g}_{j22}\right) \label{secional grupo de Lie}.
\end{eqnarray}

Now we introduce a type of asymptotic flag curvature for left invariant $p$-Finsler structures on $G$. 
Since we are using the Hamiltonian formalism on Lie groups, this curvature will depend on $\mathfrak{a} \in \mathfrak{g}^\ast\backslash \{0\}$. 

Consider $\mathfrak{a} \in S_{F_\ast}$ and $v_1 \in S_{F_e}$ that maximizes $\mathfrak{a}$ on $S_{F_e}$. 
Let $\sigma $ be a two-dimensional subspace of $\mathfrak{g}^\ast$ containing $v_1$.
Notice that $\ker \mathfrak{a}$ is transversal to $\spann \{v_1\}$ due to $\mathfrak{a}(v_1) = 1$.
Then we can choose $v_2$ in the one-dimensional space $\sigma \cap \ker \mathfrak{a}$ and we have $\sigma = \spann \{v_1, v_2\}$.
The vector $v_1$ is called the {\em flagpole} and $v_2$ is the {\em transverse edge} of the flag $\{ v_1, \sigma\}$.
For the sake of simplicity, we denote $(v_1, v_2):=\{ v_1, \sigma\}$.

The face of $S_{F_e}$ that maximizes $\mathfrak{a}$ is contained in the affine subspace $\ker \mathfrak{a} + \{v_1\}$.
We will introduce the asymptotic expansion of the flag curvature of $(v_1, v_2)$ with respect to $\mathfrak{a}$.
Let $\{ v_2, \ldots , v_n\}$ be a basis of $\ker \mathfrak{a}$.
Then $\mathcal B := \{v_1, \ldots, v_n\}$ is a basis of $\mathfrak{g}$ associated to a left invariant Riemannian metric $g$ on $G$ such that $\{v_1, \ldots, v_n\}$ is an orthonormal basis of $g\vert_\mathfrak{g}$.
Now we are ready to introduce the curvatures associated to this setting.

\begin{definition}
\label{definicao curvatura assintotica}
	Consider $\mathcal{B}$ given in the previous paragraph. 
	Define for each $k\in\mathbb{N}^*$ the basis $\mathcal{B}^{k}:=\{v_1,kv_2,\cdots,kv_n\}$ of $\mathfrak{g}$. Consider $g_{k}$ the unique left invariant Riemannian metric that makes $\mathcal{B}^{k}$ an orthonormal basis of $\mathfrak{g}$. 
	Define the function associated to the flag $(v_1,v_2)$
	$$\kappa_{g_k}(v_1,kv_2)=\kappa_{g_k}(v_1,v_2)$$ 
	which depends on $k$. We call it the asymptotic expansion of flag curvature of $(v_1, v_2)$ in $(\mathfrak{g}, F_e)$ with respect to $\mathcal{B}$.
\end{definition}

\begin{remark}
\label{observacao nao precisa Finsler}
Notice that $\kappa_{g_k}(v_1,v_2)$ is calculated using (\ref{secional grupo de Lie}), which is an object of Riemannian geometry.
Therefore no element of Finsler geometry is required.
\end{remark}

\begin{remark}
\label{observacao esfera converge a hiperplano}
	Observe that as $k \rightarrow \infty$, the unit sphere of $(\mathfrak g,g_k)$ converges locally to $\{v_1\} + \ker \mathfrak{a}$. 
	In fact, consider the coordinate system $(x^1, \ldots, x^n)$ on $\mathfrak{g}$ with respect to the basis $(v_1, \ldots, v_n)$.
	The affine subspace $\{v_1\} + \ker \mathfrak{a}$ is the graph of the function $(x^2, \ldots, x^n) \rightarrow x^1$ given by $x^1 = 1$.
	The unit sphere of $(\mathfrak{g},g_k)$ is the ellipsoid given by
	\[
	(x^1)^2 + \frac{\sum_{i=2}^n (x^i)^2}{k^2} = 1,
	\]
which is the graph of the function
	\begin{equation}
	\label{funcao elipsoide}
	x^1 = \sqrt{1 - \frac{\sum_{i=2}^n (x^i)^2}{k^2}}
	\end{equation}
	in a neighborhood of $v_1 = (1,0, \ldots, 0)$. 
	Notice that if $k \rightarrow \infty$, then (\ref{funcao elipsoide}) converges uniformly to $x^1 = 1$ in a neighborhood of $0 \in \spann\{v_2, \ldots, v_n\}$.
 \end{remark}

Let $s,t,u \in \{2, \ldots, n\}$. Then it is straightforward that
\begin{equation}
	\label{estruturas assintotica}
	\begin{array}{ccc}\alpha^{g_{k}}_{1s1}=k \,\alpha^{g}_{1s1},& \alpha^{g_{k}}_{1st}=\alpha^{g}_{1st}, & \alpha^{g_{k}}_{st1}=k^2 \alpha^{g}_{st1} \text{ and }\alpha^{g_{k}}_{stu}=k \alpha^{g}_{stu}. 
	\end{array}
\end{equation}

We split the analysis of $\kappa_{g_k}$ in two cases: $n=2$ and $n\geq 3$. 

If $n=2$, then (\ref{estruturas assintotica}) and $\alpha^g_{ijk}=-\alpha_{jik}^g$ give 
\begin{align}
\label{equacao polinomio dimensao 2} \kappa_{g_k}(v_1,v_2) = -(\alpha^{g_k}_{121})^2 -(\alpha^{g_k}_{122})^2 = -k^2 (\alpha^{g}_{121})^2 - (\alpha^{g}_{122})^2.
\end{align}

For $n\geq 3$, if we split (\ref{secional grupo de Lie}) in the cases $j=1$ and $j\geq 2$ and consider (\ref{estruturas assintotica}), we get 
\begin{align}
\kappa_{g_k}(v_1,v_2) & = k^4 \sum_{j=3}^n \frac{1}{4}(\alpha^g_{2j1})^2 && \nonumber \\
& -k^2 \left( \alpha^g_{121} \right)^2 + k^2 \sum_{j=2}^n \left( \alpha^g_{2j1} \left(\frac{1}{2}\alpha^g_{12j} - \frac{1}{2}\alpha^g_{j12} \right) - \alpha^g_{j11} \alpha^g_{j22} \right) \nonumber && \\
& + \sum_{j=2}^n \alpha^g_{12j} \left(-\frac{3}{4}\alpha^g_{12j}+\frac{1}{2}\alpha^g_{j12} \right) + \frac{1}{4}\left( \alpha^g_{j12}\right)^2.\label{equacao polinomio dimensao 3} 
\end{align}
In both cases $n=2$ and $n\geq 3$, $\kappa_{g_k} (v_1, v_2)$ is a polynomial in $k$.

\begin{definition}
\label{definition highest degree}
The coefficient of the highest degree term of $\kappa_{g_k}(v_1,v_2)$ is denoted by $\mathcal{K}_{\mathcal{B}}(\mathfrak{a})$.
\end{definition}

\begin{proposition}
\label{proposicao maior quociente}
\begin{equation}
\nonumber
\mathcal{K}_{\mathcal{B}}(\mathfrak{a}) = -(\alpha^g_{121})^2 \;\;\text{ if }\;\; n=2
\end{equation}
and
\begin{equation}
\nonumber
\mathcal{K}_{\mathcal{B}}(\mathfrak{a}) = \sum_{j=3}^n \frac{1}{4}(\alpha^g_{2j1})^2 \;\; \text{ if }\;\; n\geq 3.
\end{equation}
\end{proposition}

\begin{definition}
Let $S$ be a subset of $\mathfrak{g}$.
The normalizer of $S$ in $\mathfrak{g}$ is the subset $N_{\mathfrak{g}}(S) = \{\mathfrak{b} \in \mathfrak{g}; [\mathfrak{b},S] \subset S\}$ of $\mathfrak{g}$. 
\end{definition}

\begin{proposition}
\label{proposicao k fresco e algebrico}
Suppose that $n\geq 3$. Then $\mathcal{K}_{\mathcal{B}}(\mathfrak{a}) = 0$ iff $v_2 \in \ker \mathfrak{a} \cap N_{\mathfrak{g}}(\ker \mathfrak{a})$.
\end{proposition}

\begin{proof}
\begin{align}
\nonumber
\mathcal{K}_{\mathcal{B}}(\mathfrak{a}) = \frac{1}{4}\sum_{j=3}^n (\alpha^g_{2j1})^2 = 0
\end{align}
iff $[v_2,v_j] \subset \ker \mathfrak{a}$ for every $j=2, \ldots, n$.
Therefore $\mathcal{K}_{\mathcal{B}}(\mathfrak{a})=0$ iff $v_2 \in \ker \mathfrak{a} \cap N_\mathfrak{g}(\ker \mathfrak{a})$.
\end{proof}

\begin{remark}
\label{observacao depende so de v2}
Notice that the equality $\mathcal{K}_{\mathcal{B}}(\mathfrak{a}) = 0$ doesn't depend on $v_1$ or on the choice of $\mathcal{B}$ due to Proposition \ref{proposicao k fresco e algebrico}. 
It only depends on $\mathfrak{a}$ and $v_2$. 
From now on, we denote $\mathcal{K}_{\mathcal{B}}(\mathfrak{a})=0$ by $\mathcal{K}(\mathfrak{a},v_2) = 0$.
\end{remark}

\section{Uniqueness of Solutions of the Extended Geodesic Field for Lie groups}
\label{uniqueness}

For $p$-Finsler structures, the control $u\in \sfe$ that maximizes $\mathfrak a \in \mathfrak g^\ast$ is not necessarily unique, as it is in the strictly convex case. 
If $\mathfrak a(t)$ is the vertical part of a Pontryagin extremal, then the uniqueness of $u(t)$ such that $(u(t), \mathfrak{a}(t))$ is a Pontryagin extremal is important because $\mathfrak a(t)$ would be enough to define an extremal. 
This section is dedicated to finding relationships between this problem and the measure of 
\[
\mathcal{I} = \{ t \in I; \text{ there exist } v_2(t) \in \ker \mathfrak{a}(t)\backslash \{0\} \text{ such that } \mathcal{K}(\mathfrak{a}(t),\cdot)=0\}.
\] 

\begin{remark}
\label{observacao dual permanece na esfera}
	Notice that since $ r=\mathcal{M}(\mathfrak{a}(t))=\max_{u\in S_{F_e}}\mathfrak{a}(u)$ is constant due to Pontryagin Maximum Principle, then $\mathfrak{a}(t)\in S_{F_*}[0,r]\subset\mathfrak{g}^*$. 
\end{remark}

\begin{definition}
\label{definition infinitely many ut}
	A vertical part of a Pontryagin extremal $\mathfrak{a} (t)$ admits infinitely many $u(t)$ such that $(u(t), \mathfrak{a}(t))$ is a Pontryagin extremal if there are infinitely many admissible controls $u(t)\in\mathcal{C}(\mathfrak{a}(t))$ such that $\mathfrak{a}'(t)=-\ad^*(u(t))(\mathfrak{a}(t))$ for a.e. $t$. Likewise, we can define when $\mathfrak{a}(t)$ admits a unique solution $u(t)$ for $\mathfrak{a}^\prime (t) = -\ad^*(u(t))(\mathfrak{a}(t))$. Here we consider two controls $u(t)=\underline{u}(t)$ a.e. as the same.
\end{definition}

\begin{remark}
	\label{infinito dois}
	Notice that $\mathfrak{a}(t)$ admits infinitely many solutions $u(t)$ for $\mathfrak{a}'(t)=-\ad^*(u(t))\mathfrak{a}(t)$ iff $\mathfrak{a}(t)$ admits at least two different solutions for $\mathfrak{a}'(t)=-\ad^*(u(t))\mathfrak{a}(t)$. In fact, if there are two different admissible controls $u(t)$ and $\underline{u}(t)$ such that 
	\[
	\mathfrak{a}'(t)=-\ad^*(u(t))\mathfrak{a}(t)
	\] 
	and 
	\[
	\mathfrak{a}'(t)=-\ad^*(\underline{u}(t))\mathfrak{a}(t),
	\]   
	then $\tilde{u}(t)=su(t)+(1-s)\underline{u}(t)$ also maximizes $\mathfrak{a}(t)$ for every $s\in[0,1]$ and $\mathfrak{a}'(t)=-\ad^*(\widetilde{u}(t))\mathfrak{a}(t)$.
\end{remark}

\begin{remark}
\label{observacao nao tem a ver com extremal singular}
If $(u(t), \mathfrak{a}(t))$ and $(\underline u(t), \mathfrak{a}(t))$ are two different Pontryagin extremals, then the corresponding extremals $x_u(t)$ and $x_{\underline u}(t)$ in $G$ satisfying $x_u(t_0) = x_{\underline u}(t_0)$ and (\ref{equacao recupera trajetoria}) are different. 
Although we have multiple controls for a given $\mathfrak{a}(t)$, Definition \ref{definition infinitely many ut} doesn't have any direct relationship with singular extremals.
\end{remark}

Let us present this problem for $n=2$. 
Here $\mathfrak{g}$ can be the abelian Lie algebra or else the unique non-abelian two-dimensional Lie algebra.

If $\mathfrak{g}$ is the abelian algebra, then $\kappa_{g_k}$ is trivially zero.
If $(u(t), \mathfrak{a}(t))$ is a solution of (\ref{sistema na algebra}), then $\mathfrak{a}(t)= \mathfrak{a}_0$ for some $\mathfrak{a}_0 \in \mathfrak{g}^\ast \backslash \{0\}$.
It is straightforward that there are infinitely many solutions $(u(t),\mathfrak{a}_0)$ with vertical part $\mathfrak{a}_0$ iff $\mathfrak{a}_0$ admits infinitely many $u \in S_{F_e}$ that maximizes $\mathfrak{a}_0$, that is, $\mathfrak{a}_0$ is maximized by an edge in $S_{F_e}$.

If $\mathfrak{g}$ is the unique non-abelian two-dimensional Lie algebra, then choose a basis $\{e_1, e_2\}$ of $\mathfrak{g}$ such that $[e_1,e_2]=-e_1$.
Direct calculations shows that $\mathcal{K}(\mathfrak{a},v_2)=0$ iff $v_2$ is proportional to $e_1$, that is, $\mathfrak{a}$ is a non-zero multiple of $dx^2$.
In this case, $\mathfrak{a}^\prime(t) = 0$ and $\mathfrak{a}(t) = \mathfrak{a}$.
If $S_{F_e}$ admit an edge $E$ parallel to the $x$-axis that maximizes $\mathfrak{a}$, then any measurable control $u(t) \in E$ is such that $(u(t), \mathfrak{a})$ is a Pontryagin extremal
(compare with \cite{Gribanova}).
Therefore the two dimensional case shows evidences that given a vertical part of a Pontryagin extremal $\mathfrak{a}(t)$, the existence of infinitely many controls $u(t)$ such that $(u(t), \mathfrak{a}(t))$ is a Pontryagin extremal is related to the equation $\mathcal{K}(\mathfrak{a}(t),v_2(t)) = 0$.

From now on suppose that $n \geq 3$.
Let $\mathfrak{a}(t)$ be a vertical part of a Pontryagin extremal. 
Theorem \ref{multiplas solucoes 1} gives a necessary condition in order to $\mathfrak{a}(t)$ admit infinitely many $u(t)$ such that $(u(t),\mathfrak{a}(t))$ is a Pontryagin extremal and Theorem \ref{multiplas solucoes 2} gives a sufficient condition.

\begin{theorem}
\label{multiplas solucoes 1} Let $\mathfrak{a}(t) \in S_{F_\ast}$ be the vertical part of a Pontryagin extremal.
Suppose that $u(t)$ and $\tilde{u}(t)$ are different measurable functions such that $(u(t), \mathfrak{a}(t))$ and $(\tilde{u}(t), \mathfrak{a}(t))$ are Pontryagin extremals.
Then the set 
\begin{equation}
\nonumber 
\mathcal{I}:=\{t \in I; \text{ there exist } v_2(t)\in \ker \mathfrak{a}(t)\backslash \{0\} \text{ such that }\mathcal{K}(\mathfrak{a}(t),v_2(t)) = 0\}
\end{equation}
has positive measure.
\end{theorem}

\begin{proof}
Denote 
\[
\mathcal{I}^\prime = \{t\in I;u(t) - \tilde u(t) \neq 0\},
\]
which has positive measure by hypothesis.
Notice that 
\begin{equation}
\mathfrak{a}(t) ([u(t) - \tilde u(t), \cdot]) = \mathfrak{a}^\prime (t) - \mathfrak{a}^\prime (t) = 0,
\end{equation}
which implies that
$u(t) - \tilde{u}(t) \in \ker \mathfrak{a}(t) \cap N_{\mathfrak{g}}(\mathfrak{a}(t))$ and $\mathcal{K}(\mathfrak{a}(t), u(t) -\tilde{u}(t))=0$ for every $t\in I$ due to Proposition \ref{proposicao k fresco e algebrico}.
Moreover $v_2(t) : = u(t) - \tilde u(t) \neq 0$ for every $t \in \mathcal{I}^\prime$.
Therefore $\mathcal{I} \supset \mathcal{I}^\prime $ also has positive measure.
\end{proof}

\begin{theorem}
\label{multiplas solucoes 2}
Let $L$ be a $k$-dimensional face of $S_{F_e}$ which maximizes the relative interior of the ($n-k-1$)-dimensional face $L^\ast$ in $S_{F_\ast}$ (see Remark \ref{remark face que maximiza funcional}).
Let $(u(t), \mathfrak{a}(t))$ be a Pontryagin extremal of $(G,F)$ such that $u(t) \in L$ and $\mathfrak{a}(t) \in L^\ast$ for every $t \in I$.
Suppose that there exist a measurable function $w(t)$ parallel to $L$ such that $u(t) + w(t) \in L$ and a set of positive measure $\mathcal{I} \subset I$ such that for every $t\in \mathcal{I}$, 
\begin{itemize}
\item $w(t)\neq 0$;
\item $\mathcal{K}(\mathfrak{a}(t),w(t)) = 0$.
\end{itemize}
Then $(u(t)+w(t), \mathfrak{a}(t))$ is also a Pontryagin extremal of $(G,F)$ and $\mathfrak{a}(t)$ admits infinitely many controls $\tilde{u}(t)$ such that $(\tilde{u}(t), \mathfrak{a}(t))$ is a Pontryagin extremal.  
\end{theorem}

\begin{proof}
We claim that $\mathfrak{a}(t)([w(t),\cdot]) = 0$.
In fact, for each $t\in I$, an arbitrary $v(t) \in \mathfrak{g}$ can be written as $a(t)u(t) + b(t) z(t)$, where $a(t), b(t) \in \mathbb R$ and $z(t) \in \ker \mathfrak{a}(t)$.
Then
\begin{align}
\mathfrak{a}(t)([w(t), v(t)]) & = -a(t).\mathfrak{a}(t)([u(t),w(t)]) + b(t). \mathfrak{a}(t)([w(t),z(t)]) \nonumber \\
& = - a(t).\mathfrak{a}^\prime(t)(w(t)) + b(t).\mathfrak{a}(t)([w(t),z(t)]).\label{equacao awv}
\end{align}
The first term of the right-hand-side of (\ref{equacao awv}) is zero because $\mathfrak{a}^\prime(t)$ is parallel to $L^\ast$, $w(t)$ is parallel to $L$ and $\mathfrak{b}(\tilde{v})=1$ for every $\mathfrak{b} \in L^\ast$ and $\tilde{v} \in L$.
The second term of the right-hand-side of (\ref{equacao awv}) is zero because 
$\mathcal{K}(\mathfrak{a}(t),w(t)) = 0$, which is equivalent to the condition $w(t) \in \ker \mathfrak{a}(t) \cap N_{\mathfrak{g}}(\ker \mathfrak{a})$ due to Proposition \ref{proposicao k fresco e algebrico}.
Then $[w(t),z(t)] \in \ker \mathfrak{a}$ and the claim holds.
Therefore
\begin{equation} 
\nonumber
\mathfrak{a}^\prime (t) = \mathfrak{a}(t)([u(t), \cdot]) = \mathfrak{a}(t)([u(t)+w(t),\cdot]),
\end{equation}
and $(u(t)+w(t),\mathfrak{a}(t))$ is also a Pontryagin extremal.
Finally $\mathfrak{a}(t)$ admits infinitely many $\tilde{u}(t)$ such that $(\tilde{u}(t), \mathfrak{a}(t))$ is a Pontryagin extremal due to Remark \ref{infinito dois}.
\end{proof}

From now on, we restrict our study to the case of three-dimensional Lie groups endowed with a left invariant $p$-Finsler structure $F$. We call two-dimensional faces of $S_F \subset \mathfrak{g}$ (or else of $S_{F_\ast} \subset \mathfrak{g}^\ast$) simply by face, the one-dimensional faces by edges and the zero-dimensional faces by vertices.

\begin{proposition}
	\label{proposition ka zero kera subalgebra}
	Let $G$ be a three-dimensional Lie group endowed with a $p$-Finsler structure $F$ and $\mathfrak{a} \in S_{F_\ast}$.
	The following statements are equivalent:
	\begin{enumerate}
		\item $\mathcal K(\mathfrak{a},v_2) = 0$ for some $v_2 \in \ker \mathfrak{a}$;
		\item $\ker \mathfrak a$ is a subalgebra of $\mathfrak g$;
		\item $\mathcal K(\mathfrak{a},v) = 0$ for every $v \in \ker \mathfrak{a}$;
		\item $[v,\ker \mathfrak{a}] \subset \ker \mathfrak{a}$ for every $v \in \ker \mathfrak{a}$.
		\item There exists $v \in \ker \mathfrak{a}\backslash \{0\}$ such that $[v,\ker \mathfrak a] \subset \ker \mathfrak{a}$;
	\end{enumerate}
\end{proposition}

\begin{proof}
First of all we choose a basis $\mathcal{B}=(v_1, v_2, v_3)$ such that $v_1$ maximizes $\mathfrak{a}$ on $\sfe$ and $\ker \mathfrak{a} = \spann \{v_2, v_3\}$.
Then
\begin{equation}
\nonumber \mathcal{K}(\mathfrak{a},v_2) = \frac{1}{4}\left( \alpha^g_{231} \right)^2,
\end{equation}
where
\begin{equation}
\nonumber [v_2,v_3] = \alpha^g_{231}v_1 + \alpha^g_{232}v_2 + \alpha^g_{233}v_3
\end{equation}
and Item 1 is equivalent to Item 2. 
The equivalence between Item 2 and Item 3 follows from the fact that Item 2 doesn't depend on $v_2 \in \ker \mathfrak{a}$.
The equivalence between Item 2 and Item 4 is straightforward.
Item 4 $\Rightarrow$ Item 5 is trivial.
Finally, if Item 5 holds, consider $w \in \ker \mathfrak{a}$ such that $\{w,v\}$ is a basis of $\ker\mathfrak{a}$. 
We have that $[w,v] \in \ker \mathfrak{a}$ by Item 5, which implies that $\ker \mathfrak{a}$ is a subalgebra of $\mathfrak{g}$ and consequently Item 4 holds.
\end{proof}

\begin{remark}
\label{caso dimensao 3 k fresco nao depende v2}
Proposition \ref{proposition ka zero kera subalgebra} states that the equality $\mathcal K({\mathfrak{a}},v_2)=0$ doesn't depend on $v_2 \in \ker \mathfrak{a}$.
Therefore in the three-dimensional case we can simply denote $\mathcal{K}(\mathfrak{a}) = \mathcal{K}(\mathfrak{a},v_2)$.
\end{remark}

\begin{remark}
	\label{at em int face}
	If $(u(t),\mathfrak{a}(t))$ is a Pontryagin extremal and $t_0$ is such that $\mathfrak{a}(t_0)$ lies in the relative interior of a face of $S_{F_\ast}$, then there exists $\varepsilon >0$ such that $\mathcal{C}(\mathfrak{a}(t))= \{u_0\}$ is a vertex of $S_F$ for $t \in (t_0 - \varepsilon, t_0 + \varepsilon)$. 
	Therefore the problem of uniqueness of $u(t)$ is trivial in this case. We will suppose from here that $\mathfrak{a}(t)$ lies on an edge of $S_{F_\ast}$.
\end{remark}

\begin{theorem}
	\label{teorema principal}
	Let $G$ be a three dimensional connected Lie group endowed with a left invariant polyhedral Finsler structure $F$.
	Let $L$ be an edge of $S_{F_e}$ that maximizes the relative interior of the edge $L^\ast$ of $S_{F_\ast}$.
	Let $\mathfrak a: I \rightarrow L^\ast$ be the vertical part of a Pontryagin extremal $(u(t), \mathfrak{a}(t))$, where $u(t) \in L$. 
	Then $\mathfrak a(t)$ admits infinitely many $\tilde u(t) \in L$ such that $(\tilde u(t), \mathfrak a(t))$ is a Pontryagin extremal iff $\mathcal I: = \{t \in I; \mathcal K (\mathfrak{ a}(t)) = 0\}$ has positive measure.
\end{theorem}

\begin{proof}
If $\mathfrak{a}(t)$ admits infinitely many $u(t) \in L$ such that $(u(t), \mathfrak{a}(t))$ is a Pontryagin extremal, then $\mathcal{I}$ has positive measure due to Theorem \ref{multiplas solucoes 1}.
{Reciprocally, let $w \in \mathfrak g$ be a non-zero vector parallel to $L$. 
Recall that $\mathcal{I} = \{t \in I; \mathcal K (\mathfrak{ a}(t), w) = 0\}$ due to Proposition \ref{proposition ka zero kera subalgebra}. If $\mathcal{I}$ has positive measure, then consider  a measurable function $\lambda: I \rightarrow \mathbb R$ such that:
	\begin{itemize}
		\item $\lambda (t) \neq 0$ for $t \in \mathcal I$;
		\item $\lambda(t) = 0$ for $t \not \in \mathcal I$;
		\item $u(t)+\lambda(t).w \in L$.
	\end{itemize}  
	For instance, if $\bar u$ is the midpoint of $L$ and $r_L>0$ is such that $L$ is the line segment $[ \bar u -r_L . w, \bar u + r_L.w]$, then we can define $\lambda(t)$ as 
	\[
	\lambda(t) = 
	\left\{
	\begin{array}{ccc}
	0 & \text{ if }& t \not\in \mathcal{I}; \\
	r_L & \text{ if }& u(t) \in \mathcal{I} \cap [\bar u - r_L.w,\bar u]; \\
	-r_L & \text{ if } & u(t) \in \mathcal{I} \cap [\bar u, \bar u+r_L.w].
	\end{array}
	\right.
	\] 
	Therefore $u_\lambda(t) = u(t) + \lambda(t).w$ is another control such that $(u_\lambda(t), \mathfrak a(t))$ is a Pontryagin extremal due to Theorem \ref{multiplas solucoes 2} and there are infinitely many controls $\tilde u(t) \in L$ such that $(\tilde u(t),\mathfrak a(t))$ is a Pontryagin extremal due to Remark \ref{infinito dois}. }
\end{proof}

\begin{theorem}
	\label{teorema principal 2}
	Let $G$ be a three dimensional connected Lie group endowed with a left invariant polyhedral Finsler structure $F$.
	Suppose that $\mathfrak{a}_0$ is a vertex of $S_{F_\ast}$ and that the constant map $\mathfrak a(t) = \mathfrak{a}_0$ is the vertical part of a Pontryagin extremal of $(G,F)$. 
	Then $\mathfrak{a}_0$ admits infinitely many $u(t)$ such that $(u(t), \mathfrak{a}_0 )$ is a Pontryagin extremal iff $\mathcal K(\mathfrak{a}_0)=0$.
\end{theorem}

\begin{proof}
Suppose that $\mathfrak{a}_0$ is maximized by the face $L = \Psi^{-1}(\mathfrak{a})\subset S_{F_e}$.
If $(u(t),\mathfrak{a}_0)$ is a Pontryagin extremal of $(G,F)$, then we have necessarily that $u(t) \in L$. 
If $\mathfrak{a}(t) = \mathfrak{a}_0$ admits infinitely many $u(t)$ such that $(u(t), \mathfrak{a}(t))$ is a Pontryagin extremal, then $\mathcal{K}(\mathfrak{a}_0)=0$ due to Theorem \ref{multiplas solucoes 1}. 

Reciprocally, if $\mathcal{K}(\mathfrak{a}_0)=0$, then any measurable function $\tilde u(t) \in L$ such that $\tilde{u}(t) - u(t) \neq 0$ can be written as $\tilde{u}(t) = u(t) + (\tilde{u}(t) - u(t))$, where $w(t)=\tilde{u}(t) - u(t)$ satisfies all conditions for $w(t)$ of Theorem \ref{multiplas solucoes 2} for $\mathcal{I}=I$.
Thus $\mathfrak{a}_0$ admits infinitely many controls $\tilde u(t)$ such that $(\tilde u(t), \mathfrak{a}_0)$ is a Pontryagin extremal.
\end{proof}

\section{Final remarks}
\label{final-remarks}

In this section, we make remarks about this work and give suggestions for future works.

Theorems \ref{teorema principal} and \ref{teorema principal 2} can be used in order to try to classify the left invariant $p$-Finsler structures on three-dimensional Lie groups such that for every vertical part $\mathfrak{a}(t)$ of a Pontryagin extremal, there exist a unique $u(t)$ such that $(u(t), \mathfrak{a}(t))$ is a Pontryagin extremal.

\label{paragrafo explica ponto conjugado} In Finsler geometry, if we consider a point $p \in M^n$ and a geodesic $\gamma$ such that $\gamma(0)=p$, then the first conjugate point on $\gamma$ can be estimated by the flag curvatures along $\gamma$. 
Moreover the geodesic is no longer minimizing beyond this conjugate point. 
In fact, Ricci curvature can be defined for Finsler manifolds in a similar way that it is defined for Riemannian manifolds, that is, as a sum of $(n-1)$ flag curvatures (see \cite[Section 7.6]{BaoChernShen}).
If $(n-1)\rho$ is a positive lower bound  for the Ricci curvature, then the distance of two conjugate points in $M$ is less than or equal to $\pi / \sqrt{\rho} $ (see \cite[Section 7.7]{BaoChernShen}).
Eventually such an analysis can be tried on Lie groups with a left invariant $p$-Finsler structure with an approximation of $F_e$ by Riemannian metrics.
Another question is whether the property of $\ker \mathfrak{a}$ being a subalgebra in higher dimensions has something to do with geodesic properties of $G$. 
\label{paragrafo geodesic properties} For instance, if $(u(t), \mathfrak{a}(t))$ is a Pontryagin extremal of $G$ such that $\ker \mathfrak{a}(t)$ is a subalgebra of $\mathfrak{g}$ for every $t$ and $x(t)$ is the extremal in $G$ correspondent to $(u(t), \mathfrak{a}(t))$, does it implies that $x(t)$ is a geodesic?.

Finally we hope that Theorems \ref{teorema principal} and \ref{teorema principal 2} would help us to understand and eventually classify the dynamics of the control system $\mathfrak a^\prime(t) = -\ad^\ast (u(t))(\mathfrak a(t))$ on three-dimensional Lie groups endowed with a left invariant $p$-Finsler structure.
				
\section{Acknowledgements}
				
The authors would like to thank Professors Alexandre José Santana, Claudio Aguinaldo Buzzi, Hugo Murilo Rodrigues, Lino Grama and Maria Elenice Rodrigues Hernandes for their suggestions. 
The authors also would like to thank the suggestions given by reviewers that were very helpful for this work.
The first author would like to thank the Ph.D. fellowship from the Brazilian National Council for Scientific and Technological Development, CNPq, Brazil.

\end{document}